\documentclass[12pt]{amsart}
\usepackage[a4paper, left=3cm, right=3cm, height=22cm]{geometry}
\usepackage{amsmath,amsfonts,amsthm, amssymb}
\usepackage[arrow,matrix, all,cmtip]{xy}
\usepackage{enumerate}

\usepackage[usenames]{color}
\usepackage{amssymb}
\newtheorem{theorem}{Theorem}[section]
\numberwithin{equation}{section}
\newtheorem{proposition}[theorem]{Proposition}

\newtheorem{corollary}[theorem]{Corollary}

\theoremstyle{remark}

\newtheorem{example}[equation]{Example}
\newtheorem{remark}[theorem]{Remark}

\newtheorem{conjecture}[equation]{Conjecture}

\newcommand{\mb}{\mathbb}

\newcommand{\ml}{\mathcal}

\newcommand{\mr}{\mathrm}

\def\C{{\mathbb{C}}}

\def\mL{{\mathbb{L}}}
\def\P{{\mathbb{P}}}
\def\Q{{\mathbb{Q}}}
\def\Z{{\mathbb{Z}}}

\def\Ch{{\rm Ch}}

\def\sp{{\rm SP}}
\def\Griff{{\rm Griff}}

\def\mZ{{\mathcal{Z}}}

\usepackage[pdfauthor={Jin Cao}, %
				pdftitle={Cdga}%
			dvips,colorlinks=true]{hyperref}
\hypersetup{
	bookmarksnumbered=true,
	linkcolor=black,
}
\newcounter{elno}

\begin{document}
\author{Jin Cao and Wenchuan Hu}
\title{Topological and Geometric filtration for products}
\date{\today}

\address{
Jin Cao\\
Yau Mathematical Sciences Center,
Tsinghua University, Beijing, China
}
\email{caojin@mail.tsinghua.edu.cn}

\address{Wenchuan Hu,
School of Mathematics,
Sichuan University,
Chengdu,
China
}
\email{huwenchuan@gmail.com}

%%========================================================%%%
%%%        Affiliation
%%========================================================%%%

\begin{abstract}
 We show that the Friedlander-Mazur conjecture holds for the product of an elliptic curve with some smooth projective variety of dimension $3$. Moreover, we show that the Friedlander-Mazur conjecture is stable under a surjective map. As applications, we show that  the Friedlander-Mazur conjecture holds uniruled threefolds and unirational varieties up to  certain range.
\end{abstract}
\maketitle
\pagestyle{myheadings}
 \markright{Lawson homology}

\tableofcontents

\section{Introduction}

In this paper, all varieties are defined over the complex number field $\C$.
Let $X$ be a complex projective variety of dimension $n$.

%We denote by $\mZ_p(X)$ the space of algebraic $p$-cycles on $X$. Let $\Ch_p(X)$ be the Chow group of $p$-cycles on $X$, i.e.,
%$\Ch_p(X)=\mZ_p(X)/{\rm \{rational ~equivalence\}}$.

% Let $H_k(X,\Q):=H_k(X)\otimes \Q$ be the singular homology with rational coefficients.
%We denote by $p_g(X)$ the geometric genus of $X$, defined to be $\dim H^n(X,\mathcal{O}_X)$.
%Let $q(X)$ be the irregularity of $X$, defined to be $\dim H^1(X,\mathcal{O}_X)$.

%$\mr{NS}(X)$: the Neron-Severi group of $X$.
%Let $L_pH_k(X)$ be the $(p,k)$-Lawson homology group on $X$ which we will recall below.

The \emph{Lawson homology}
$L_pH_k(X)$ of $p$-cycles for a projective variety is defined by
$$L_pH_k(X) := \pi_{k-2p}({\mathcal Z}_p(X)) \quad for\quad k\geq 2p\geq 0,$$
where ${\mathcal Z}_p(X)$ is the space of algebraic $p$-cycles on $X$ provided with a natural topology (see
\cite{Friedlander1}, \cite{Lawson1}). It has been extended to define for a quasi-projective  variety by Lima-Filho (see \cite{Lima-Filho})
and Chow motives (see \cite{Hu-Li}). For the general background, the
reader is referred to Lawson' survey paper \cite{Lawson2}.
%The definition of Lawson homology has been extended to negative integer $p$. Formally for $p<0$, we have
%$L_pH_k(X)=\pi_{k-2p}({\mathcal Z}_{0}(X\times \mathbb{C}^{-p}))=H^{BM}_{k-2p}(X\times\mathbb{C}^{-p})=H^{BM}_k(X)=L_0H_k(X)$
%(cf. \cite{Friedlander-Haesemesyer-Walker}), where $H^{BM}_{*}(-)$ denotes the Borel-Moore homology.

In \cite{Friedlander-Mazur}, Friedlander and Mazur showed that there
are  natural transformations, called \emph{Friedlander-Mazur cycle class maps}
\begin{equation}\label{eq01}
\Phi_{p,k}:L_pH_{k}(X)\rightarrow H_{k}(X)
\end{equation}
for all $k\geq 2p\geq 0$.

Recall that Friedlander and Mazur constructed a map called the $s$-map $$s:L_pH_k(X)\to L_{p-1}H_k(X)$$ such that the cycle class map
$\Phi_{p,k}=s^p$ (see \cite{Friedlander-Mazur}).
Explicitly, if $\alpha\in L_pH_k(X)$ is represented by the homotopy class of a continuous map  $f:S^{k-2p}\to \mZ_p(X)$, then
$\Phi_{p,k}(\alpha)=[f\wedge S^{2p}]$, where $S^{2p}=S^2\wedge\cdots\wedge S^2$ denotes the $2p$-dimensional topological sphere.

Set
{$$
\begin{array}{llcl}
&L_pH_{k}(X)_{hom}&:=&{\rm ker}\{\Phi_{p,k}:L_pH_{k}(X)\rightarrow H_{k}(X)\};\\
&L_pH_{k}(X, \Q)&:=&L_pH_{k}(X)\otimes\Q;\\
&T_pH_{k}(X)&:=&{\rm Image}\{\Phi_{p,k}:L_pH_{k}(X)\rightarrow
H_{k}(X)\};\\
&T_pH_{k}(X,{\mathbb{Q}}) &:=&T_pH_{k}(X)\otimes {\mathbb{Q}}.
\end{array}
 $$}

For simplicity, the map $ \Phi_{p,k}\otimes{\Q}:L_pH_{k}(X)_{\Q}\rightarrow H_{k}(X,\Q)$
is also denoted by $\Phi_{p,k}$.
%The \emph{Griffiths group} of dimension $p$-cycles is defined to be
%$${\rm Griff}_p(X):={\mathcal Z}_p(X)_{hom}/{\mathcal Z}_p(X)_{alg}.$$

%Set
%$$
%\begin{array}{lcl}
%{\rm Griff}_p(X)_{\Q}&:=&{\rm Griff}_p(X)\otimes\Q;\\
%{\rm Griff}^q(X)&:=&{\rm Griff}_{n-q}(X);\\
%{\rm Griff}^q(X)_{\Q}&:=&{\rm Griff}_{n-q}(X)_{\Q}.
%%\end{array}
%$$

%It was proved by Friedlander \cite{Friedlander1} that, for any
%smooth projective variety $X$, $$L_pH_{2p}(X)\cong {\mathcal
%Z}_p(X)/{\mathcal Z}_p(X)_{alg}=A_p(X).$$

%Therefore
%\begin{eqnarray*}
%L_pH_{2p}(X)_{hom}\cong {\rm Griff}_p(X).
%\end{eqnarray*}

\iffalse
For any smooth quasi-projective variety
$X$, there is an intersection pairing (cf. \cite{Friedlander-Gabber})  $$L_pH_k(X)\otimes L_qH_l(X)\to L_{p+q-n}H_{k+l-2n}(X),$$
induced by the diagonal map $\Delta: X\to X\times X$. More precisely, the composition $\mZ_p(X)\times \mZ_q(X)\stackrel{\times}{\to}\mZ_{p+q}(X\times X)\stackrel{\Delta^{!}}{\to} Z_{p+q-n}(X)$, where $\times$ is the Cartesian
product of cycles and $\Delta^{!}$ is the Gysin map, factors through $\mZ_p(X)\wedge \mZ_q(X)$. On the level of homotopy groups we have intersection pairing
$$
\pi_{k-2p}(\mZ_p(X))\otimes\pi_{l-2q}(\mZ_q(X))\stackrel{\bullet}{\to}\pi_{k+l-2(p+q)}(\mZ_{p+q-n}(X)),
$$
that is,
$$L_pH_k(X)\otimes L_qH_l(X)\stackrel{\bullet}{\to} L_{p+q-n}H_{k+l-2n}(X).
$$
\fi

It was shown in \cite[\S 7]{Friedlander-Mazur} that the subspaces
$T_pH_k(X,{\mathbb{Q}})$ form a decreasing filtration (called the \emph{topological filtration}):
$$\cdots\subseteq T_pH_k(X,{\mathbb{Q}})\subseteq T_{p-1}H_k(X,{\mathbb{Q}})
\subseteq\cdots\subseteq
T_0H_k(X,{\mathbb{Q}})=H_k(X,{\mathbb{Q}})$$ and
$T_pH_k(X,{\mathbb{Q}})$ vanishes if $2p>k$.

Denote by
$C_pH_k(X,{\mathbb{Q}})\subseteq H_k(X,{\mathbb{Q}})$ the
$\mathbb{Q}$-vector subspace of $H_k(X,{\mathbb{Q}})$ spanned by
the images of correspondence homomorphisms $\phi_Z: H_{k-2p}(Y,{\mathbb{Q}})\rightarrow
H_k(X,{\mathbb{Q}})$, as $Y$ ranges through all smooth projective varieties of dimension $k-2p$
and $Z$ ranges all algebraic cycles on $Y\times X$ equidimensional over $Y$ of relative dimension $p$.

Denote by
$G_pH_k(X,{\mathbb{Q}})\subseteq H_k(X,{\mathbb{Q}})$ the
$\mathbb{Q}$-vector subspace of $H_k(X,{\mathbb{Q}})$ generated by
the images of mappings $H_k(Y,{\mathbb{Q}})\rightarrow
H_k(X,{\mathbb{Q}})$, induced from all morphisms $Y\rightarrow X$ of
varieties of dimension $\leq k-p$.

The subspaces $G_pH_k(X,{\mathbb{Q}})$ also form a decreasing
filtration (called the \emph{geometric filtration}):
$$\cdots\subseteq G_pH_k(X,{\mathbb{Q}})\subseteq G_{p-1}H_k(X,{\mathbb{Q}})
\subseteq\cdots\subseteq G_0H_k(X,{\mathbb{Q}})\subseteq
H_k(X,{\mathbb{Q}})$$

%Denote by $\tilde{F}_pH_k(X,{\Q})\subseteq
%H_k(X,{\Q})$  the maximal sub-(Mixed) Hodge structure of span
%$k-2p$. (See \cite{Grothendieck2} and \cite{Friedlander-Mazur}.) The sub-${\Q}$ vector spaces
%$\tilde{F}_pH_k(X,{\Q})$ form a decreasing filtration of
%sub-Hodge structures:
%$$\cdots\subseteq \tilde{F}_pH_k(X,{\Q})\subseteq \tilde{F}_{p-1}H_k(X,{\Q})
%\subseteq\cdots\subseteq \tilde{F}_0H_k(X,{\Q})\subseteq H_k(X,{\Q})$$ and $\tilde{F}_pH_k(X,{\Q})$ vanishes if $2p>k$. This
%filtration is called the \emph{Hodge filtration}.

It was shown by Friedlander and Mazur that
\begin{equation}\label{eq20}
T_pH_k(X,{\mathbb{Q}})=C_pH_k(X,{\mathbb{Q}})\subseteq G_pH_k(X,{\mathbb{Q}}) %\subseteq \tilde{F}_pH_k(X,{\Q})
\end{equation}
holds for any smooth projective variety $X$ and $k\geq 2p\geq0$.

One can also define the Hodge filtration on
$H_k(X,{\Q})$ as follows: Denote by
$\tilde{F}_pH_k(X,{\Q})\subset H_k(X,{\Q})$ to be the
maximal sub-Mixed Hodge structure of span $k-2p$.(See \cite{Friedlander-Mazur} for example.)  The sub-${\Q}$ vector spaces $\tilde{F}_pH_k(X,{\Q})$ form a decreasing filtration of sub-Hodge structures:
$$\cdots\subset \tilde{F}_pH_k(X,{\Q})\subset \tilde{F}_{p-1}H_k(X,{\Q})
\subset\cdots\subset \tilde{F}_0H_k(X,{\Q})\subset H_k(X,{\Q})$$ and $\tilde{F}_pH_k(X,{\Q})$ vanishes if $2p>k$. This is
the homological version of Hodge filtration. The we have:
\begin{equation}\label{eq21}
T_pH_k(X,{\mathbb{Q}})=C_pH_k(X,{\mathbb{Q}})\subseteq G_pH_k(X,{\mathbb{Q}}) \subseteq \tilde{F}_pH_k(X,{\Q})
\end{equation}

Friedlander and Mazur proposed the following conjecture which closely relates Lawson homology theory to  the Grothendieck Standard conjecture B in the algebraic cycle theory.
\begin{conjecture}[Friedlander-Mazur conjecture, \cite{Friedlander-Mazur}]\label{conj9.1}
Let $X$ be a smooth projective variety. Then one has
$$
T_pH_k(X,{\mathbb{Q}})= G_pH_k(X,{\mathbb{Q}})
$$
for $k\geq 2p\geq0$.
\end{conjecture}

It has been shown in \cite{Friedlander2} that the Grothendieck Standard conjecture B holds for all smooth projective varieties  implies that
the Friedlander-Mazur conjecture holds for all smooth projective varieties. The inverse implication has been shown in \cite{Beilinson}.
However, it is still an open problem at this moment whether  the Grothendieck Standard conjecture B holds for $X$ and the Friedlander-Mazur conjecture
 holds for $X$ are equivalent or not for a given smooth projective variety $X$.

The Friedlander-Mazur conjecture holds for smooth projective varieties of dimension less than or equal to two but it remains open for threefolds in general.
 However, it has been verified  for some cases in dimension three or above.
For example, it holds for cellular varieties for which the Lawson homology are shown to be the same as the singular homology (see \cite{Lawson1}, \cite{Lima-Filho});
it holds for general abelian varieties (see \cite{Friedlander2}) or abelian varieties for which the  generalized Hodge conjecture holds (see\cite{Abdulali});
it holds for smooth projective varieties for which the Chow groups in rational coefficients are isomorphic to the corresponding singular homology groups in rational coefficients
by using the technique of decomposition of diagonal (cf. \cite{Voineagu}).
It was also shown to hold for threefold $X$ with $h^{2,0}(X)=0$ (see \cite{Hu}).
 It also holds for \emph{any} abelian threefold.
This and a survey  of these materials, including its relation to the Grothendieck Standard conjecture B and the Generalized Hodge conjecture, can be found in the appendix of \cite{Hu2}.

The main purpose in this paper is to show that the Friedlander-Mazur conjecture holds for a sequence of projective varieties in low dimensions.
In Theorem \ref{Thm3.5} the Friedlander-Mazur conjecture is shown to hold for the product of  a smooth projective curve and a smooth projective surface. In Theorem \ref{Thm2.7}, the Friedlander-Mazur conjecture is shown to hold for the product of  an elliptic curve and a smooth projective variety of dimension three which satisfies $H^{3,0}=0$, the Grothendieck standard conjecture B and generalized Hodge conjectures. %In Theorem \ref{Thm3.10},
 %the Friedlander-Mazur conjecture is shown to hold for the product of arbitrary number of smooth projective curves.
 We also show that the
 Friedlander-Mazur conjecture is stable under the surjective morphism (see Proposition \ref{Prop6.2}).
% As a corollary, we show that  the Friedlander-Mazur conjecture  holds for all Jacobian varieties of smooth projective curves and the moduli space of stable vector bundles ofcoprime rank and degree over any smooth projective curve.
We show that the  Friedlander-Mazur conjecture  holds for uniruled threefolds and smooth unirational varieties of arbitrary dimension in certain range of indices. We also observe that the motivic invariants
 can not be used to distinguish  rational varieties and the unirational varieties.

\section{Friedlander-Mazur conjecture}

In this section we will show that the Friedlander-Mazur conjecture holds for the product of  a smooth projective curve and a smooth projective surface and the product of some projective threefold and an ellipit curve. For a smooth projective threefold $X$,
the surjection of  the cycle class map $\Phi_{1,4}\otimes \Q: L_1H_4(X)_{\Q} \to H_4(X,\Q)$  is  equivalent to a positive answer to
many questions including the Grothendieck standard conjecture B  in the algebraic cycle theory. %Then we apply the idea further to prove that the Frielander-Mazur conjectureholds for the products of smooth projective curves.

\begin{proposition}\label{Prop4.1}
Let $X = C \times S$, where $C$ is a smooth projective curve and $S$ is a smooth projective surface. Then $\Phi_{1,4}: L_1H_4(X, \Q)  \to H_4(X,\Q)$ is surjective.
\end{proposition}

\begin{proof}[The first proof of Proposition \ref{Prop4.1}]
Since $X=C\times S$, we have a map $\mZ_p(C)\wedge \mZ_q(S)\to \mZ_1(C\times S)=\mZ_1(X)$ for nonnegative integers $p,q$ such that $p+q=1$, where
$\wedge$ is the smash map. This map induces a map of Lawson homology $L_pH_k(C)\otimes L_qH_l(S)\to  L_1H_4(X)$ for nonnegative integers $k\geq 2p,l\geq 2q$, where $k+l=4$.
Moreover, this map commutes with the natural transformation $\Phi_{*,*}:L_*H_*(-)\to H_*(-)$  and hence we have the following commutative diagram
\begin{equation}\label{equ4.1}
\xymatrix{
L_pH_k(C, \Q)\otimes L_qH_l(S, \Q)\ar[r]\ar[d]& L_1H_4(X, \Q)\ar[d]\\
H_k(C, \Q)\otimes H_l(S, \Q)\ar[r]& H_4(X, \Q).\\
}
\end{equation}

Now for $\alpha\in H_4(X,\Q)$, by K\"{u}nneth formula we can find $c_{k}\in H_k(C,\Q) , s_{l}\in H_l(S,\Q)$ such that
$$\sum_{k+l=4} c_{k}\otimes s_{l}=\alpha.$$
Since $k+l=4$ and $0\leq k\leq 2$, there are the following cases:

\begin{enumerate}
\item[(1)] $k=0,l=4$. Since $\Phi_{0,0}:L_0H_0(C, \Q)\cong H_0(C,\Q)$ and  $\Phi_{1,4}:L_1H_4(S, \Q)\cong H_4(S,\Q)$ (see \cite{Friedlander1}) are isomorphisms, we have
$$L_0H_0(C, \Q)\otimes L_1H_4(S, \Q)\cong H_0(C,\Q)\otimes H_4(S,\Q).$$
Hence there exist $c_{0}'\in L_0H_0(C,\Q) , s_{4}'\in L_1H_4(S, \Q)$ such that
 $\Phi_{0,0}(c_{0}')=c_{0}$ and $\Phi_{1,4}(s_{4}')=s_{4}$. By construction, $c_{0}'\otimes s_{4}'$ maps to an element (still denote by $c_{0}'\otimes s_{4}'$) in $L_1H_4(C\times S, \Q)$ which maps to
 the K\"{u}nneth component $c_{0}\otimes s_{4}$ of $\alpha$.
\item[(2)] $k=1,l=3$.  Since $\Phi_{0,1}:L_0H_1(C, \Q)\cong H_1(C,\Q)$ (Dold-Thom Theorem) and  $\Phi_{1,3}:L_1H_3(S, \Q)\cong H_3(S,\Q)$ (see \cite{Friedlander1}) are isomorphisms, we have
$L_0H_1(C, \Q)\otimes L_1H_3(S, \Q)\cong H_1(C,\Q)\otimes H_3(S,\Q)$. Hence there exist $c_{1}'\in L_0H_1(C, \Q) , s_{3}'\in L_1H_3(S, \Q)$ such that
 $\Phi_{0,1}(c_{1}')=c_{1}$ and $\Phi_{1,3}(s_{3}')=s_{3}$. By construction, $c_{1}'\otimes s_{3}'$ maps to an element (still denote by $c_{1}'\otimes s_{3}'$) in $L_1H_4(C\times S, \Q)$ which maps to
 the K\"{u}nneth component $c_{1}\otimes s_{3}$ of $\alpha$.
\item[(3)] $k=2,l=2$. Since $\Phi_{1,2}:L_1H_2(C)_{\Q}\cong H_2(C,\Q)$  and  $\Phi_{0,2}:L_0H_2(S, \Q)\cong H_2(S,\Q)$ (Dold-Thom Theorem) are isomorphisms, we have
$$L_1H_2(C, \Q)\otimes L_0H_2(S, \Q)\cong H_2(C,\Q)\otimes H_2(S,\Q).$$
Hence there exist $c_{2}'\in L_1H_2(C, \Q) , s_{2}'\in L_0H_2(S, \Q)$ such that
 $\Phi_{1,2}(c_{2}')=c_{2}$ and $\Phi_{0,2}(s_{2}')=s_{2}$. By construction, $c_{2}'\otimes s_{2}'$ maps to an element (still denote by $c_{2}'\otimes s_{2}'$) in $L_1H_4(C\times S, \Q)$ which maps to
 the K\"{u}nneth component $c_{2}\otimes s_{2}$ of $\alpha$.
\end{enumerate}

Therefore, by the commutative diagram in Equation \eqref{equ4.1} we get an element $\sum_{k+l=4} c_{k}'\otimes s_{l}'\in L_1H_4(X, \Q)$ such that
$$\Phi_{1,4}\big(\sum_{k+l=4} c_{k}'\otimes s_{l}'\big)=\alpha.$$
This completes the proof of the surjectivity of $\Phi_{1,4}$.
\end{proof}

\begin{remark}
The idea of the proof of Proposition \ref{Prop4.1} can be traced to \cite[p.195]{Lawson2}, where the topological, geometric and Hodge filtrations of the product of $n$ elliptic curves are studied.
We remark that the idea of the proof there for the coincidence of the topological, geometric and transcendental Hodge filtrations works \emph{only} for the cases $k\geq p+n$, where $k$ is the homological dimension
and $p$ is the dimension of the cycles.
\end{remark}

\begin{remark}
From this proof we see that $\Phi_{1,4}: L_1H_4(X) \to H_4(X)$ is surjective since $L_pH_k(C)\cong H_k(C)$ for all $k\geq 2p\geq 0$,  $L_1H_4(S)\cong H_4(S)$,  $\Phi_{1,3}:L_1H_3(S)\cong H_3(S)$ and $\Phi_{0,2}:L_0H_2(S, \Q)\cong H_2(S)$.
 Hence torsion elements in $H_4(C\times S)$ come from those of $L_1H_4(C\times S)$.
\end{remark}

\begin{proof}[The second proof of Proposition \ref{Prop4.1}]
Let $X$ be any smooth projective threefold such the Grothendieck standard conjecture B holds. That is, the inverse $\Lambda$ of the Lefschetz operator $L:H^i(X,\Q)\to H^{i+2}(X,\Q)$ is an algebraic operator.
From the argument in the proof $(iv)\Rightarrow (i)$ of the proposition in \S 2.2 in \cite{Beilinson}, one observes that $L_1H_k(X, \Q)\to H_k(X,\Q)$ is surjective for all $k\geq 4$. Now taking $X=C\times S$,
we see that $X$ is smooth projective threefold satisfying   Grothendieck standard conjecture B  since the conjecture holds in dimension less than three and is stable under products (see \cite{Kleiman}).
Therefore, $L_1H_4(X, \Q)\to H_4(X,\Q)$ is surjective.
\end{proof}

\begin{corollary}\label{Cor4.3}
Let $X$ be the product of a smooth projective curve $C$ and a smooth projective surface $S$, then the Friedlander-Mazur conjecture holds for $X$.
\end{corollary}
\begin{proof}
Note that for any smooth projective threefold $X$, the Friedlander-Mazur conjecture has been proved to hold except for ``$T_1H_4(X,\Q)=G_1H_4(X,\Q)$" (see \cite{Hu}).
Now $H_4(X,\Q)=G_1H_4(X,\Q)$ holds  for $X$ by the definition of the geometric filtration and $T_1H_4(X,\Q)$ is exactly the image of $\Phi_{1,4}:L_1H_4(X, \Q) \to H_4(X,\Q)$.
Hence the surjectivity of $\Phi_{1,4}:L_1H_4(X, \Q)\to H_4(X,\Q)$ by Proposition \ref{Prop4.1} implies that ``$T_1H_4(X,\Q)=G_1H_4(X,\Q)$" for $X=C\times S$.
This completes the proof of the corollary.
\end{proof}

\begin{remark}
As a comparison, the generalized Hodge conjecture is still open even for the product of three smooth projective curves.
\end{remark}

\begin{remark}\label{Thm3.5}
Let $X$ be a smooth projective varieties of dimension three. If the Grothendieck standard conjecture B holds for $X$, then Friedlander-Mazur conjecture holds for $X$. This has been prove in \cite{Friedlander2}. Alternatively, by \cite[Remark 1.13]{Hu}, the Friedlander-Mazur conjecture  $$T_pH_k(X,{\mathbb{Q}})= G_pH_k(X,{\mathbb{Q}})$$ holds for $X$ except for $p=1,k=4$. By  the same argument as in the second proof of Proposition \ref{Prop4.1}, we obtain that $L_1H_4(X, \mathbb{Q})\to H_4(X,{\mathbb{Q}})$ is surjective under Grothendieck standard conjecture B for $X$. This completes the proof.
\end{remark}

Our next result is to provide some examples for smooth projective varieties of dimensional four whose FM conjecture holds.

\begin{theorem} \label{Thm2.7}
Let $X$ be the product of a smooth three-dimensional projective variety $Y$ with a smooth projective curve $C$ of genus $1$. Assume that $h^{3,0}(Y) = 0$, $G_1H_3(Y, \Q) = \widetilde{F}_1H_3(Y, \Q)$ and the Grothendieck standard conjecture B holds for $Y$. Then the Friedlander-Mazur conjecture holds for $X$.
\end{theorem}
\begin{proof}
For $p=0$, by the Dold-Thom theorem and the Weak Lefschetz theorem, we get the statement  ``$T_pH_k(X,\Q)=G_pH_k(X,\Q)$"  for all $k\geq 0$.

For $p=4$, the statement  ``$T_pH_k(X,\Q)=G_pH_k(X,\Q)$" holds trivially from their definitions.

For $p=3$, the statement  ``$T_pH_k(X,\Q)=G_pH_k(X,\Q)$" follows from Friedlander's computation of Lawson homology for codimension one cycles (see \cite{Friedlander1}).

For $p=2$, the statement  ``$T_2H_4(X,\Q)=G_2H_4(X,\Q)$" follows from  \cite[\S 7]{Friedlander-Mazur}; the statement  ``$T_2H_5(X,\Q)=G_2H_5(X,\Q)$" follows from \cite[Prop.1.15]{Hu};
the statement  ``$T_2H_k(X,\Q) = G_2H_k(X,\Q)$"  for $k\geq 6$ follows from  the arguments in the proof $(iv)\Rightarrow (i)$ of proposition in \cite[\S2.2]{Beilinson}.
Another proof of the last statement can be obtained in a similar way from  the first proof of Proposition \ref{Prop4.1}.

For $p=1$, as above, the statement ``$T_1H_2(X,\Q)=G_1H_2(X,\Q)$" follows from \cite[\S 7]{Friedlander-Mazur}; the statement ``$T_1H_3(X,\Q)=G_1H_3(X,\Q)$" follows from \cite[Prop.1.15]{Hu}.
Moreover, the equality
$$``T_1H_k(X,\Q) = G_1H_k(X,\Q)"$$  for $k\geq 5$ follows from  the arguments in the proof $(iv)\Rightarrow (i)$ of Proposition in \cite[\S2.2]{Beilinson} or by a direct check
as that in Proposition \ref{Prop4.1}.

The remain case is  the statement ``$T_1H_4(X,\Q)=G_1H_4(X,\Q)$".  In fact, we have the following commutative diagram:
\begin{small}
\begin{equation} \label{com diag}
\xymatrix{
L_1H_4(Y) \otimes H_0(C) \oplus L_1H_3(Y) \otimes H_1(C) \oplus L_1H_2(Y) \otimes H_2(C)\ar[r]\ar[d]& L_1H_4(X)\ar[d]\\
G_1H_4(Y) \otimes H_0(C) \oplus G_1H_3(Y) \otimes H_1(C) \oplus G_1H_2(Y) \otimes H_2(C)\ar[r]\ar[d]& G_1H_4(X)\ar[d]\\
\widetilde{F}_1(H_4(Y) \otimes H_0(C)) \oplus \widetilde{F}_1(H_3(Y) \otimes H_1(C)) \oplus \widetilde{F}_1(H_2(Y) \otimes H_2(C))\ar[r]& \widetilde{F}_1H_4(X)\\
}
\end{equation}
\end{small}
Note that the first left vertical map is sujective due the Grothendieck standard conjecture B holds for $Y$.
\

Give an element $\alpha \in i_*H_4(V, \Q) \subset G_1H_4(X, \Q)$, where $i: V \to X$ and $V$ is a three-dimensional smooth projective variety. View $\alpha$ as an element in $\widetilde{F}_1H_4(X)$.
\

Consider the Kunneth decomposition of $\alpha$ in the Hodge filtration $H_4(X, \Q)$. Then $\alpha = \alpha_0 + \alpha_1 + \alpha_2$, where $\alpha_i \in H_{4-i}(Y, \Q) \otimes H_i(C, \Q)$. Notice that the standard conjecture holding for $Y$ implies that $L_1H_4(X, \Q) \to G_1H_4(X, \Q) \to H_4(X, \Q)$ is surjective. See the proof of Theorem 2.6.  Also the map $G_1H_2(X, \Q) \to H_2(X, \Q)$ is surjective. Therefore $\alpha_0$ (resp. $\alpha_2$) belongs to $G_1H_4(Y, \Q) \otimes H_0(C, \Q)$ (resp. $G_1H_2(Y, \Q) \otimes H_2(C, \Q)$).
\

%Because there are natrual isomorphisms between $H_3(C)$ and $H^3(C)$ (resp. $H_4(X)$ and $H^4(X)$), the usual cup product structure on cohomology induces a product structure on homology, which we also denote by $\cup$ for simplicity.

%We let $\gamma_1, \gamma_2$ be the generators of $H_1(C, \Q)$ \begin{color}{red} such that $\gamma_{1}^{-1,0} = \gamma_2^{0,-1}$ when $\gamma_1, \gamma_2$ are considered as elements in $H_1(C, \mb{C}) = H_{-1,0} \oplus H_{0,-1}$. Here we use the homological Hodge index. See Chapter 2.10 of \cite{Friedlander2}. \end{color} Then there exist $\beta_1, \beta_2 \in H_3(Y, \Q)$ such that $\alpha_1 = \beta_1 \otimes \gamma_1 + \beta_2 \otimes \gamma_2 \in G_1H_4(X, \Q)$. \begin{color}{red} We let $\gamma \in H_{-1,0}(C, \C)$ define by:
%\[
%\gamma = \frac{1}{2}(\gamma_1 + i \gamma_2), \bar{\gamma} =  \frac{1}{2}(\gamma_1 - i \gamma_2).
%\]
%\end{color}
We let $\gamma_1, \gamma_2$ be the generators of $H_1(C, \Q)$ such that
\[
\gamma = \frac{1}{2}(\gamma_1 + i \gamma_2) \in H_{-1,0} , \bar{\gamma} =  \frac{1}{2}(\gamma_1 - i \gamma_2) \in H_{0,-1},
\]
where $H_1(C, \mb{C}) = H_{-1,0} \oplus H_{0,-1}$ and we use the homological Hodge index (See \cite[\S 7.2]{Friedlander-Mazur}).
Then there exist $\beta_1, \beta_2 \in H_3(Y, \Q)$ such that $\alpha_1 = \beta_1 \otimes \gamma_1 + \beta_2 \otimes \gamma_2 \in G_1H_4(X, \Q)$ and hence
\begin{equation}
\begin{split}
& \alpha_1 = \beta_1 \otimes \gamma_1 + \beta_2 \otimes \gamma_2 = \beta_1 \otimes (\gamma+\bar{\gamma}) + \beta_2 \otimes \frac{1}{i}(\gamma - \bar{\gamma})\\
= & (\beta_1 - i \beta_2) \otimes \gamma + (\beta_1+i \beta_2) \otimes \bar{\gamma}
\end{split}
\end{equation}
Because $\alpha_1 \in G_1H_4(X, \Q)$ doesn't contain $(-4,0)$ and $(0,-4)$ component,  $(\beta_1 - i \beta_2)^{-3,0} = (\beta_1+i \beta_2)^{0,-3} = 0$. If we set $\beta = \frac{1}{2}(\beta_1 - i \beta_2)$ and $\bar{\beta} = \frac{1}{2}(\beta_1 + i \beta_2)$, then we find that:
\[
\alpha_1 = 2 \beta \otimes \gamma + 2 \bar{\beta} \otimes \bar{\gamma},
\]
where $\beta$ (resp. $\bar{\beta}$) doesn't have $(-3,0)$ (resp. $(0,-3)$)-component. On the other hand, we have $i_*H_4(V, \C) \cap H_{0,-3}(Y) \otimes H_{-1,0}(C) = \{0\}$ because $h^{3,0}(Y) = 0$.%\begin{color}{red}On the other hand, $i_*H_4(V, \C) \cap H_{0,-3}(Y) \otimes H_{-1,0}(C) = \{0\}.$ %In fact, if the image $\pi: V \to Y \times C \to C$ is a point, then $\pi_*(i_*H_4(V, \C)) \subset H_0(C, \C)$. Or if the image $\pi: V \to Y \times C \to C$ is surjective, then $\pi_*(i_* H_4(V, \C)) \subset H_2(C, \C)$. From the Hodge indices, both are impossible.\end{color}
Therefore $\mr{Span}_{\Q}\{\alpha_1\} \otimes \C \in (F_1H_3(Y, \C) \otimes H_1(C, \C)) \cap (H_3(Y, \Q) \otimes H_1(C, \Q))$. Moreover $\mr{Span}_{\Q}\{\beta_1, \beta_2\} \otimes \C = \mr{Span}_{\C}\{\beta, \bar{\beta}\}$ carries a sub Hodge structure of $H_3(Y, \C)$. This implies that $\alpha_1 \in  \mr{Span}_{\C}\{\beta, \bar{\beta}\} \otimes H_1(C, \mb{C}) \subset \widetilde{F}_1(H_3(Y)) \otimes H_1(C, \mb{C})$ and therefore $\alpha_1 \in G_1H_3(Y, \Q) \otimes H_1(C, \Q)$ by our assumption, which further implies that the middle horizontal map in (\ref{com diag}) is surjective. Then $L_1H_4(X) \to G_1H_4(X)$ is surjective, which is equivalent to saying that $T_1H_4(X,\Q)=G_1H_4(X,\Q)$.
\end{proof}

\begin{example}
in \cite{Tankeev}, Tankeev shows that the Grothendieck standard conjecture B folds for all smooth complex projective threefolds of Kodaira dimension small than $3$. Hence via the above theorem and Remark \ref{Thm3.5}, the Friedlander-Mazur conjecture holds for the product of an elliptic curve and any smooth complex projective threefolds of Kodaira dimension small than $3$ satisfying $h^{3,0}(Y) = 0, G_1H_3(Y, \Q) = \widetilde{F}_1H_3(Y, \Q)$. For example, the Friedlander-Mazur conjecture holds for the product of an elliptic curve with a complete intersection of dimension three.
\end{example}

%%%%%%%%%%%%%%%%%%%%%%%%%%%%%%%%%%%%%% surjectivity
\section{Surjective morphisms}
In this section, we let $\ml{M}$ be the category of Chow motives with rational coefficients. For any smooth projective variety $X$ over the complex number field $\mb{C}$, we denote its Chow motive by $h(X)$.
Recall that Vial shows the following result.
\begin{theorem}[ \cite{Vial2}]\label{Thm7.1}
Let $f: X \to B$ be a surjective morphism of smooth projective varieties over $\C$. Then $\bigoplus^{d_X - d_B}_{i = 0}h(B)(i)$ is a direct summand of $h(X)$.
\end{theorem}

A natural question is that how does the Friedlander-Mazur conjecture behave under  surjective morphisms.  We have the following result which
says that the  Friedlander-Mazur conjecture respects surjective morphisms.

\begin{proposition}\label{Prop6.2}
Let $f: X \to B$ be a surjective morphism of smooth projective varieties over $\C$. If the Friedlander-Mazur conjecture holds for $X$, then it holds for $B$.
\end{proposition}
\begin{proof}
By Theorem 4.7 in \cite{Hu-Li}, we have
 $$
 \begin{array}{ccl}
 L_pH_k(\bigoplus^{d_X - d_B}_{i = 0} h(B)(i), \Q)&=&\bigoplus^{d_X - d_B}_{i = 0} L_pH_k(h(B)(i), \Q)\\
 &=&\bigoplus^{d_X - d_B}_{i = 0} L_{p-i}H_{k-2i}(B, \Q).
 \end{array}
 $$

By Theorem \ref{Thm7.1}, $\bigoplus^{d_X - d_B}_{i = 0} L_{p-i}H_{k-2i}(B)$ is a direct summand of $L_pH_k(X)$, where we note
that the choice of Leschetz motive $\mL$ in \cite{Hu-Li} is the inverse to that in \cite{Vial2}. Since
 the singular homology also respects the direct sum decomposition of motives, the image of the natural transform $\Phi_{p,k}$ on
 each summand of $h(X)$ lies in the corresponding summand of its singular homology.
 Now by assumption, the Friedlander-Mazur conjecture holding for $X$ means $$``T_pH_k(X,\Q) = G_pH_k(X,\Q)"$$
for $k\geq 2p$. Therefore, we have
$$T_pH_k(\bigoplus^{d_X - d_B}_{i = 0} h(B)(i),\Q)=G_pH_k(\bigoplus^{d_X - d_B}_{i = 0} h(B)(i),\Q),$$
that is,
$$\bigoplus^{d_X - d_B}_{i = 0} T_{p-i}H_{k-2i}(B,\Q)= \bigoplus^{d_X - d_B}_{i = 0} G_pH_{k-2i}(B,\Q).$$
Hence $T_{p-i}H_{k-2i}(B,\Q)= G_pH_{k-2i}(B,\Q)$ for each $ 0\leq i\leq d_X-d_B$.
In particular, we have $T_pH_k(B,\Q) = G_pH_k(B,\Q)$ for all $k\geq 2p$.
This completes the proof of the proposition.
\end{proof}
 
\begin{remark}
Let $f: X \to B$ be a surjective morphism of smooth projective varieties over $\C$. If the Grothendieck Standard conjecture holds for $X$, then it holds for $B$.
This can be deduced from \cite[Lemma 4.2]{Arapura}. From the proof of Proposition \ref{Prop6.2}, one can see that  the generalized Hodge conjecture also respects to
 surjective morphisms.

Since $f^*$ commutes with the Lefschetz operator $L:H^k(X)\to H^{k+2}(X)$, we have the following commutative diagram:
$$
\xymatrix{H^k(B)\ar[r]^{\Lambda_B}\ar[d]^{f^*}&H^{k-2}(B)\ar[d]^{f^*}\\
H^k(X)\ar[r]^{\Lambda_X}&H^{k-2}(X).
}
$$
By Theorem \ref{Thm7.1}, $f^*$ on $H^{k-2}(B)$ is injective and its image ${\rm im}(f^*)\subset H^{k-2}(X)$  is  the summand
of $H^{k-2}(X)$. In other words, the assumption that $\Lambda_X$ is algebraic implies that its components are algebraic.
\end{remark}

\section{Unirational and uniruled varieties}
Recall that a smooth projective variety $X$ is called \emph{unirational} if there is a positive integer $n$ such that $f: \mb{P}^n \dashrightarrow X$ is a dominant rational map.
\begin{proposition}\label{prop41}
Let $X$ be a smooth unirational variety of dimension $n$. Then we have $L_1H_k(X)_{hom} \otimes \mb{Q} = 0$ for any integer $k\geq 2$. Furthermore if $k \geq 2\dim X$, then $L_1H_k(X)_{hom} = 0.$
That is, $L_1H_{2n}(X) = \Z$ and  $L_1H_k(X)= 0$ for $k> 2n$. Similar results holds for if 1-cycles are replaced by codimension 2-cycles.

\end{proposition}
\begin{proof}
It has been shown in \cite[Prop. 6.6]{Hu-Li} that $L_1H_k(X)_{hom} \otimes \mb{Q} = 0$ since that $L_1H_k(\P^n)_{hom} \otimes  \mb{Q}=0$.  Therefore $L_1H_k(X)_{hom}$
must be torsion elements.  Since $f$ is a finite map, the degree $d$ in a positive integer. Hence the element $\alpha\in L_1H_k(X)_{hom}$ satisfies that
$d\alpha=0$. Since  $L_1H_k(X)_{hom}$ is divisible for $k\geq 2\dim X$ (see \cite[Prop. 3.1]{Voineagu}), we get $\alpha=0$.
The proof for codimension 2 cycles is similar.
\end{proof}

\begin{remark}
This is a generalization of a result  in \cite{Hu-Li}, where either the dimension of $X$ is not more than four or the group rational coefficient. A different method using
the decomposition of diagonal can be found in  \cite{Peters} and \cite{Voineagu}.
\end{remark}

Given a smooth projective variety $Y$ of dimension $n-2$ and a point $e\in Y$, we put
$\textbf{p}_0=e\times Y$ and $\textbf{p}_2=Y\times e$, then take
$\textbf{p}_1=\Delta_Y-\textbf{p}_0-\textbf{p}_2$ where $\Delta_Y$
is the diagonal in $Y\times Y$. Then we have $h(Y)=h(pt)\oplus
\mL\oplus Y^+=1\oplus
\mL\oplus Y^+$, where $\mL=h(pt)(-1)$ is the Lefschetz motive and
$Y^+=(Y,id-\textbf{p}_0-\textbf{p}_2)$.

\begin{corollary}\label{Coro7.4}
Let $X$ be a unirational variety of $\dim X=n$. Then the motive $h(X)$ can be written as
$$
h(X)=1\oplus a\mL\oplus U\otimes \mL \oplus a\mL^{n-1}\oplus \mL^{n},
$$
where $U$ a direct summand of a motive of the form $\oplus Y_i^+$, the $Y_i$'s being smooth projective
varieties of dimension  equal to $n-2$, $a\in \Z^+$ and $a\mL$ means the direct sum of $\mL$ for $a$ times.
\end{corollary}
\begin{proof}
Since $X$ is unirational of dimension $n$, there exists a dominant rational map $\P^n\dashrightarrow X$. By  a sequence of blow ups along
codimension at least 2 smooth subvarieties, we can obtain a finite surjective morphism $\widetilde{\P}^n\to X$. By Theorem \ref{Thm7.1}
we obtain that $h(X)$ is a direct summand of  $h(\widetilde{\P}^n)$, which by the blow up formula is of the form $h(\P^n)\oplus (\oplus_i  h(Y_i)(-1))$. Note
that if the dimension of the blow up center is less than $n-2$, the additional part generated by the blowup can also be written in the form $h(Y)(-1)$, where $Y$ is smooth projective and $\dim Y=n-2$.
Since $ h(Y_i)(-1)=h(Y_i)\otimes \mL=(1\oplus \mL^{n-2}\oplus Y_i^+)\otimes \mL$, $h(\widetilde{\P}^n)=1\oplus b\mL\oplus (\oplus_i  h(Y_i^+))\otimes \mL \oplus b\mL^{n-1}\oplus \mL^{n}$ for
some $b\in \Z^+$. Hence $h(X)=1\oplus a\mL\oplus U\otimes \mL \oplus a'\mL^{n-1}\oplus \mL^{n}$,  where $U$ a direct summand of a motive of the form $\oplus_i Y_i^+$ and $a, a'\in \Z^+$.
By Poincar\'{e} duality, $a=a'$. This completes the proof of the corollary.
\end{proof}

\begin{remark}
By Corollary \ref{Coro7.4}, we observe that there is no big difference between the rationality and unirationality in the sense of Chow motives. More precisely, we are not able to
determine whether a unirational variety $X$ is rational or not through computing their  invariants which are realizations of Chow motives,
such as its singular homology group with rational coefficients, Chow groups with rational coefficients, Lawson homology groups with rational coefficients.
That is, there exist two unirational varieties $X, Y$ such that $h(X)\cong h(Y)$, where $X$ is  rational but $Y$ is not.
\end{remark}

\begin{corollary}
Let $X$ be a unirational variety of $\dim X=n$. Then
$T_1H_k(X,\Q) = G_1H_k(X,\Q)$ for all $k\geq 2$ and $T_{n-2}H_k(X,\Q) = G_{n-2}H_k(X,\Q)$ for $k\geq 2(n-2)$.
\end{corollary}
\begin{proof} Since $X$ is unirational variety of $\dim X=n$, there is a finite surjective morphism $\widetilde{\P}^n\to X$, where $\widetilde{\P}^n$ is
a sequence of blow  ups along codimension at least 2 smooth subvarieties. Since the statement ``$T_1H_k(Y,\Q) = G_1H_k(Y,\Q)$" is a birational statement for
a smooth projective variety $Y$ (see \cite{Hu}), we have ``$T_1H_k(\widetilde{\P}^n,\Q) = G_1H_k(\widetilde{\P}^n,\Q)$". Now the corollary follows from Proposition
\ref{Prop6.2}.  The proof of the statement $T_{n-2}H_k(X,\Q) = G_{n-2}H_k(X,\Q)$ is similar. This completes the proof of the corollary.
\end{proof}

%\begin{remark}
%As a comparison, the kernel $\ker(\Phi_{p,k})=L_pH_{k}(X)_{hom}$ for $p=1$ and $p=\dim X-2$ to a unirational variety $X$  have been studied in \cite{Hu-Li} in a similar technique.
%\end{remark}

Recall that a smooth projective variety $X$ of dimension $n$ is called \emph{uniruled},  that is, a there is a smooth projective variety $Y$ of dimension $n-1$
such that $f: \mb{P}^1\times Y \dashrightarrow X$ is a dominant rational map.
\begin{proposition}\label{Prop4.6}
Let $X$ be a smooth uniruled threefold. Then we have $$L_pH_k(X)_{hom} \otimes \mb{Q} = 0$$ for any integer $k\geq 2p\geq 0$.
 Furthermore $L_pH_6(X)=\Z$ and $L_pH_k(X)= 0$ for all $0\leq p\leq 3$ and $k > 6$.
\end{proposition}

\begin{proof}
Since $ \mb{P}^1\times Y \dashrightarrow X$ is a dominant rational map, we get a surjective morphism $ \widetilde{\mb{P}^1\times Y} \to X$ by a sequence of blow ups
$\widetilde{\mb{P}^1\times Y} \to  {\mb{P}^1\times Y}$. Since $$L_pH_k(\widetilde{\mb{P}^1\times Y})_{hom}\cong L_pH_k({\mb{P}^1\times Y})_{hom}$$ (see \cite{Hu}) and
 $L_pH_k({\mb{P}^1\times Y})\cong L_{p-1}H_{k-2}(Y)\oplus  L_pH_{k}(Y)$ (see \cite{Friedlander-Gabber}), we have
 $$L_pH_k(\widetilde{\mb{P}^1\times Y})_{hom}\cong L_{p-1}H_{k-2}(Y)_{hom}\oplus  L_pH_{k}(Y)_{hom}=0$$ since $Y$ is a projective surface (see \cite{Friedlander1}).
The last statement follows from the same reason as that in Proposition \ref{prop41}.
\end{proof}

Now we turn to the Friedlander-Mazur conjecture for uniruled threefolds.
\begin{corollary}
Let $X$ be a uniruled threefold.  Then the Friedlander-Mazur conjecture holds for $X$.
That is, $T_pH_k(X,\Q) = G_pH_k(X,\Q)$ for all $k\geq 2p$.
\end{corollary}
\begin{proof}
Since $X$ is a uniruled threefold, there exists a smooth projective surface $Y$ such that
$ \mb{P}^1\times Y \dashrightarrow X$ is a dominant rational map.
By Corollary \ref{Cor4.3}, $T_pH_k(\mb{P}^1\times Y,\Q) = G_pH_k(\mb{P}^1\times Y,\Q)$.
 As in Proposition \ref{Prop4.6},
we can find a finite surjective morphism $\widetilde{\mb{P}^1\times Y} \to X$, where
$\widetilde{\mb{P}^1\times Y} \to  {\mb{P}^1\times Y}$ is a sequence of blow ups at smooth centers.
Since the  the Friedlander-Mazur conjecture holds or not is a birational invariant statement for
smooth projective varieties of dimension less than or equal to four (see \cite{Hu}), we have
$T_pH_k(\widetilde{\mb{P}^1\times Y},\Q) = G_pH_k(\widetilde{\mb{P}^1\times Y},\Q)$. Now for all $k\geq 2p$,
 the equality  $T_pH_k(X,\Q) = G_pH_k(X,\Q)$   follows from Proposition \ref{Prop6.2}.
\end{proof}

Now we introduce a notation ``unirational map". A rational map $f:X\dashrightarrow Y$ between two irreducible projective varieties
$X, Y$ of the same dimension is called  a \emph{unirational map} if $f$ is a dominant map. Then $Y$ is called a \emph{uni-X variety}.
For convenience, $X$ is always chosen as a smooth projective variety.
Note that it coincides with the notations of  unirational (resp. uniruled) variety.  The following result
is a summary of the result in this section.

\begin{proposition}
Let $Y$ be a smooth uni-$X$ variety of dimension $n$. If $$L_1H_k(X)_{hom} \otimes \mb{Q} = 0$$ for any integer $k\geq 2$, so is
for $Y$.  If $T_1H_k(X, \Q)=G_1H_k(X, \Q)$ holds for  $k\geq 2$, so is for $Y$.  Similarly,
if $L_{n-2}H_k(X)_{hom} \otimes \mb{Q} = 0$ for any integer $k\geq 2$, so is
for $Y$.  If $T_{n-2}H_k(X, \Q)=G_{n-2}H_k(X, \Q)$ holds for  $k\geq 2$, so is for $Y$.
\end{proposition}
\begin{proof}
By assumption, there is a dominant rational map $f:X\dashrightarrow Y$. By   a sequence of blow  ups along codimension at least 2 smooth subvarieties, we
get a surjective morphism $\tilde{f}:\widetilde{X}\to Y$. Since $L_1H_k(\widetilde{X})_{hom}\cong L_1H_k(X)_{hom} \otimes \mb{Q}$(see \cite{Hu}), which is $0$ by assumption.
By \cite[Prop. 6.6]{Hu-Li}, we have  $\dim_{\Q}L_1H_k(Y)_{hom,  \mb{Q}}\leq \dim_{\Q}L_1H_k(\widetilde{X})_{hom,  \mb{Q}}$. Hence we have
$L_1H_k(X)_{hom} \otimes \mb{Q} = 0$.

Since the statement ``$T_1H_k(W,\Q) = G_1H_k(W,\Q)$" is a birational statement for
a smooth projective variety $W$ (see \cite{Hu}), we have ``$T_1H_k(\widetilde{X},\Q) = G_1H_k(\widetilde{X},\Q)$". Now the statment
$T_1H_k(Y,\Q) = G_1H_k(Y,\Q)$ follows from Proposition \ref{Prop6.2} and the fact $\tilde{f}:\widetilde{X}\to Y$ is a surjective morphism.

The case of codimension 2 cycles is similar.
\end{proof}

\emph{Acknowledgements.}
The first author would like to thank Tianyuan Mathematical Center in Southwest China for their hospitality and financial support. The first author was partially supported by the National Science Foundation for Young Scientists of China (Grant No. 11901334). The second author was partially sponsored by  STF of Sichuan province, China(2015JQ0007), NSFC(11771305).
%%========================================================%%
%%              References
%%========================================================%%

\end{document}